\documentclass[a4paper,12pt]{amsart}

\usepackage{fouriernc}
\usepackage[T1]{fontenc}

\usepackage{graphicx}
\usepackage{amsmath}
\usepackage{amssymb}
\usepackage[active]{srcltx}
\usepackage{hyperref}
\usepackage{bbm}

\usepackage{geometry}
\newtheorem{thm}{Theorem}%
\theoremstyle{definition}

\theoremstyle{remark}
\theoremstyle{plain}

\def\EE{{\mathbb E}}

\def\HH{{\mathbb H}}

\def\NN{{\mathbb N}}
\def\QQ{{\mathbb Q}}
\def\PP{{\mathbb P}}
\def\RR{{\mathbb R}}
\def\TT{{\mathbb T}}

\def\ZZ{{\mathbb Z}}

\def\vecm{{\text{\boldmath$m$}}}

\def\vecq{{\text{\boldmath$q$}}}

\def\vecx{{\text{\boldmath$x$}}}

\def\vecy{{\text{\boldmath$y$}}}

\def\vecxi{{\text{\boldmath$\xi$}}}

\def\vecnull{{\text{\boldmath$0$}}}

\def\scrA{{\mathcal A}}

\def\scrL{{\mathcal L}}

\def\scrP{{\mathcal P}}

\def\e{\mathrm{e}}

\def\C{\operatorname{C{}}}

\def\L{\operatorname{L{}}}

\def\SL{\operatorname{SL}}
\def\ASL{\operatorname{ASL}}

\def\vol{\operatorname{vol}}

\def\GamG{\Gamma\backslash G}

\def\ASLZ{\ASL(2,\ZZ)}
\def\ASLR{\ASL(2,\RR)}

\def\SLZ{\SL(2,\ZZ)}
\def\SLR{\SL(2,\RR)}

\title[Square roots and lattices]{Square roots and lattices}
\author{Jens Marklof}
\address{Jens Marklof, School of Mathematics, University of Bristol, Bristol BS8 1UG, U.K.\newline \rule[0ex]{0ex}{0ex} \hspace{8pt}{\tt j.marklof@bristol.ac.uk}}
\date{13 June 2024/10 December 2024}
\thanks{Research supported by EPSRC grant EP/W007010/1. Data supporting this study are included within the article. MSC2020: 11K06,37D40,60G55}

\begin{document}

\begin{abstract}
We construct a point set in the Euclidean plane that elucidates the relationship between the fine-scale statistics of the fractional parts of $\sqrt n$ and directional statistics for a shifted lattice. We show that the randomly rotated, and then stretched, point set converges in distribution to a lattice-like random point process. This follows closely the arguments in Elkies and McMullen's original analysis for the gap statistics of $\sqrt{n}\bmod 1$ in terms of random affine lattices [Duke Math.\ J.\ 123  (2004), 95--139]. There is, however, a curious subtlety: the limit process emerging in our construction is {\em not} invariant under the standard $\SLR$-action on $\RR^2$. 
\end{abstract}

\maketitle

\section{Introduction}

{\em Square roots:} In their landmark paper \cite{Elkies04} published twenty years ago, Elkies and McMullen proved that the gap distribution for the fractional parts of $\sqrt n$ ($n=1,\ldots,N$, $N\to\infty$) converges to the previously unknown limiting distribution plotted as the continuous curve in Figure \ref{figStats} (left). To state their result more precisely, let us denote by $0\leq \xi_1\leq \xi_2\leq\ldots\leq \xi_N<1$ the fractional parts of $\sqrt n$ ($n=1,\ldots,N$) ordered by size, and furthermore let $s_n = \xi_{n+1}-\xi_n$ be the gap between $\xi_n$ and $\xi_{n+1}$ ($n=1,\ldots,N-1$), and $s_N=1+\xi_1-\xi_N$ the gap between $\xi_N$ and $1+\xi_1$ (think of the unit interval $[0,1)$ as the real line mod 1 with the end points $0$ and $1$ identified). It is important to note that $\xi_n$ and $s_n$ will change as we move to a different $N$. The Elkies-McMullen theorem then states that, for every $s>0$,
\begin{equation}\label{EM007}
\lim_{N\to\infty} \frac{\#\{ n\leq N \mid N s_n > s \}}{N} = \int_s^\infty P(s')\, ds' ,
\end{equation}
where the gap density $P(s)$ is a piecewise analytic function with a power-law tail, cf.~Figure~\ref{figStats} (left); for an explicit formula see \cite[Theorem 3.14]{Elkies04}.
The scaling of $s_n$ by $N$ is necessary in view of the average gap size $1/N$. It is remarkable that if we change the square root to any other fractional power $n^\beta$ with $0<\beta<1$, the gap distribution of the fractional parts appears to have instead an exponential limit density $P(s)=\e^{-s}$ -- the same as for a Poisson point process! This observation is purely conjectural, and mainly based on numerical experiments. The best rigorous results in this direction are currently due to Lutsko, Sourmelidis and Technau \cite{Lutsko24}, who showed that the two-point correlation function (a slightly weaker fine-scale statistics) of $n^\beta$ converges to that of a Poisson point process, when $0<\beta\leq\frac13$. We refer the reader to \cite{Lutsko24,Technau23} for further background and references on the pseudo-random properties of related arithmetic sequences.

\begin{figure}[h]
\begin{center}
\begin{minipage}{0.49\textwidth}
\unitlength0.1\textwidth
\begin{picture}(10,8)(0,0)
\put(-0.2,-1.5){\includegraphics[width=\textwidth]{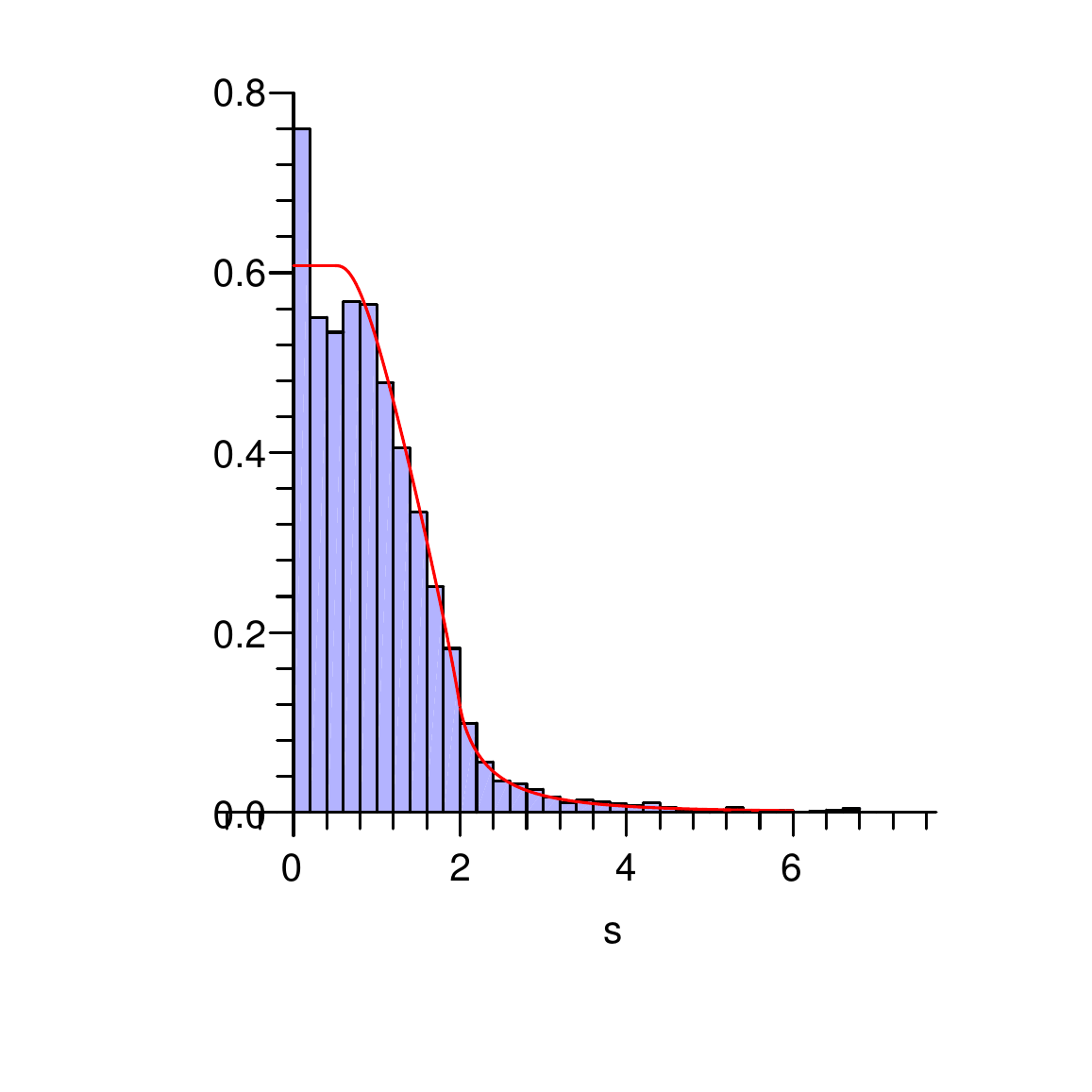}}
\end{picture}
\end{minipage}
\begin{minipage}{0.49\textwidth}
\unitlength0.1\textwidth
\begin{picture}(10,8)(0,0)
\put(-0.2,-1.5){\includegraphics[width=\textwidth]{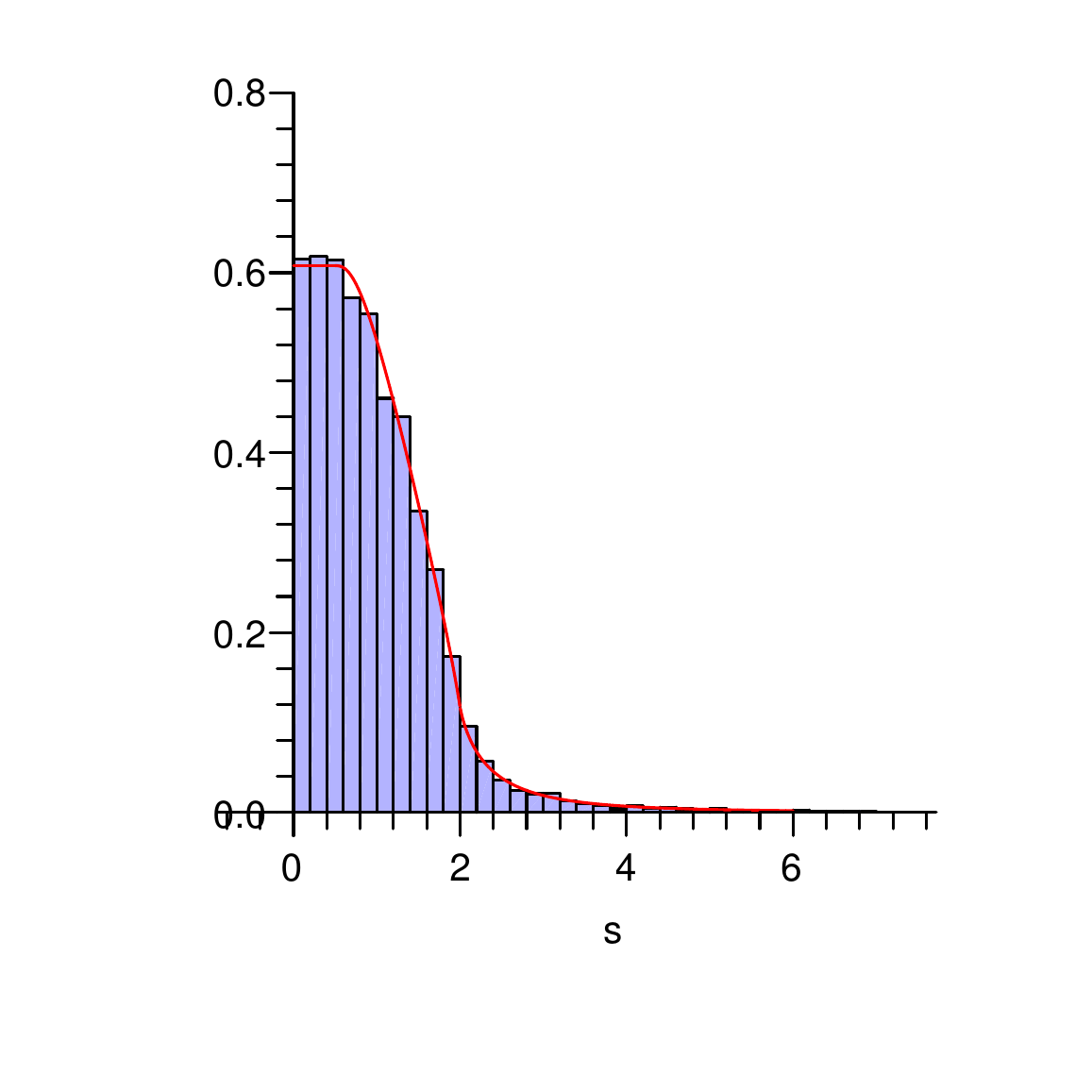}}
\end{picture}
\end{minipage}
\end{center}
\caption{The distribution of gaps in the sequence $\sqrt{n}\bmod 1$, $n=1,\ldots,7765$ (left) and in the directions of the vectors $(m-\sqrt{2},n)\in\RR^2$ with $m\in\ZZ$, $n\in\ZZ_{\geq 0}$, $(m-\sqrt2)^2+n^2< 4900$ (right). The continuous curve is the Elkies-McMullen distribution. Reproduced from \cite{partI}.}\label{figStats}
\end{figure}

{\em Lattices:} Let us now consider a given Euclidean lattice $\scrL\subset\RR^2$ of full rank; the simplest example to keep in mind is $\scrL=\ZZ^2$. We are interested in the directional statistics of the lattice points of $\scrL$ as viewed from a fixed observer located at $\vecq\in\RR^2$. Let $\vecy_1,\ldots,\vecy_N$ denote the first $N$ shortest vectors in the shifted lattice $\scrL-\vecq$ (if there are two or more vectors of the same length, pick your favourite order), and record their angles (relative to the horizontal axis, say) as $\theta_1,\ldots,\theta_N\in[0,2\pi)$. Dividing by $2\pi$ and ordering by size produces a set of $\xi_n$ in $[0,1)$, as above for the fractional parts. One of the findings of \cite{partI} is that (a) the gap distribution also converges for this new sequence, and (b) its limit density agrees with the Elkies-McMullen distribution if and only if $\vecq\notin\QQ\scrL$. This is illustrated in Figure~\ref{figStats} (right) for $\scrL=\ZZ^2$ and $\vecq=(\sqrt 2,0)$. 

The fundamental reason why the two limit distributions are the same is that they follow from the equidistribution of two different unipotent translates on the space of affine lattices which both converge to the {\em same} invariant measure, and are integrated against the {\em same} test function. In the present paper we will provide a more intuitive explanation of this surprising phenomenon by formulating the Elkies-McMullen convergence in terms of a natural point process in $\RR^2$ (which is different from the unipotent dynamics on the space of lattices considered in their original paper). The idea is to construct a point set in $\RR^2$ such that (a) its directions exactly reproduce the fractional parts of $\sqrt n$ and (b) it is approximated locally by affine lattices. 

For other aspects of the statistics of $\sqrt n\bmod 1$ we refer the reader to \cite{ElBaz15,Fraczek15,Marklof07,RadTech2024,Sinai13}; for the distribution of directions in affine lattices, see \cite{ElBaz15b,Kim24,partI,Strombergsson}; for general background and applications of spherical averages of point sets, see the introduction of \cite{MarklofVinogradov}.

\section{Random point sets}\label{RPS}

In the following, elements of $\RR^2$ are represented as row vectors. Define the $2\times 2$ matrices
\begin{equation}
D(T)=\begin{pmatrix} T^{-1/2} & 0 \\ 0 & T^{1/2}\end{pmatrix},\qquad
k(\theta)=\begin{pmatrix} \cos\theta & -\sin\theta \\ \sin\theta & \cos\theta \end{pmatrix} ,
\end{equation}
and let $\theta$ be a random variable on $\RR/2\pi\ZZ$ distributed according to an absolutely continuous probability measure $\lambda$. Note that $D(T)$ and $k(\theta)$ have unit determinant and are thus elements of the special linear group $\SLR$. For fixed $\scrL$, $\vecq$ as above, and $T>0$, the random set
$\Xi_T=(\scrL-\vecq) k(\theta) D(T)$ defines a point process in $\RR^2$, i.e., a random counting measure that assigns a unit mass to the location of every element in $\Xi_T$. By abuse of notation, we will use the same symbol for a random point set and the corresponding point process. For more background on point processes and random point sets, see \cite[Section 4]{Stoyan}. A key observation in the proof of \cite[Theorem 1.3]{partI} is that the convergence of the gap distribution for the directions follows from the convergence (in distribution) of the point processes $\Xi_T\Rightarrow\Xi$ for $T\to\infty$ to a limit process $\Xi$. It is proved in \cite{partI} (see Theorem 2.1 and Section 6) that $\Xi$ is given by a random affine lattice whose distribution will depend on the choice of $\vecq$. If $\vecq\notin\QQ\scrL$, then the limit process $\Xi$ is in fact independent of $\vecq$, and distributed according to the unique $\ASL(2,\RR)$-invariant probability measure on the space of affine lattices. Here $\ASL(2,\RR)$ denotes the semidirect product group $\SL(2,\RR)\ltimes\RR^2$ with multiplication law $(M,\vecxi)(M',\vecxi')=(MM',\vecxi M'+\vecxi')$. The (right) action of $(M,\vecxi)$ on $\RR^2$ is defined by $\vecx\mapsto \vecx M+\vecxi$, and the space of affine lattices can be identified with the homogeneous space $X=\GamG$ with $G=\ASLR$, $\Gamma=\ASLZ$. We embed $\SLR$ in $\ASLR$ by $M\mapsto (M,\vecnull)$ and use the shorthand  $M$ for $(M,\vecnull)$.

\begin{figure}
\begin{center}
\includegraphics[width=0.49\textwidth]{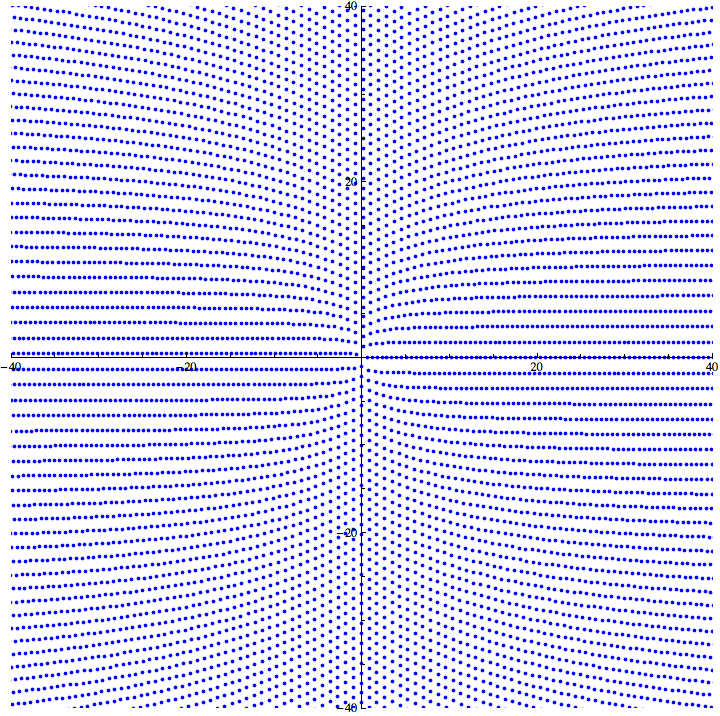}
\includegraphics[width=0.49\textwidth]{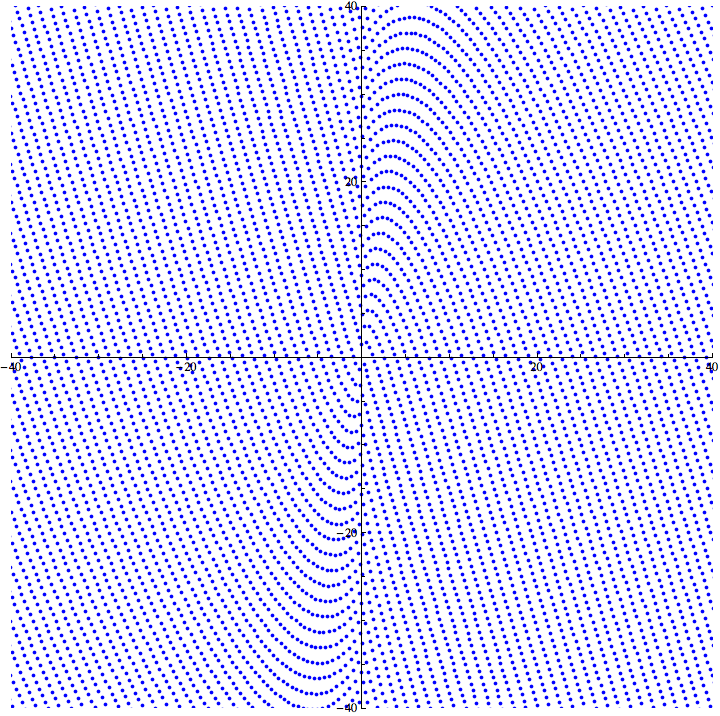}
\end{center}
\caption{{\em Left:}  The point set $\scrP$. {\em Right:}  A realization of $\Theta_T$ with $T=4$ and $\theta=0.7$.} \label{fig1}
\end{figure}

\begin{figure}
\begin{center}
\includegraphics[width=0.49\textwidth]{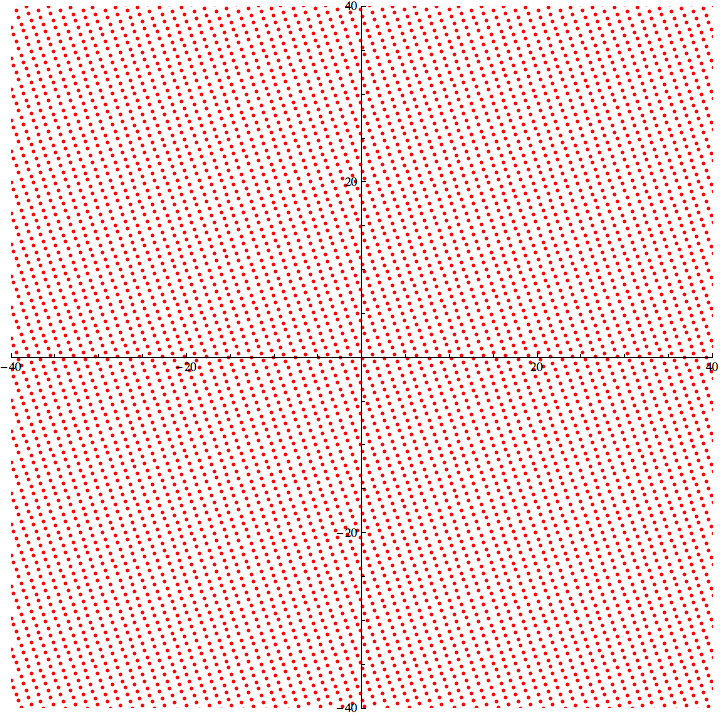}
\includegraphics[width=0.49\textwidth]{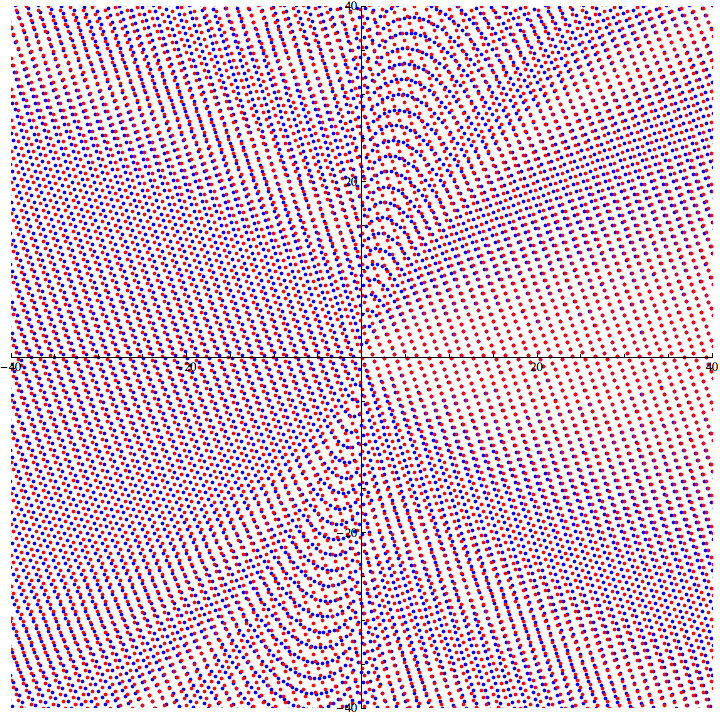}
\end{center}
\caption{{\em Left:}  A realization of the lattice $\Xi_T$ in \eqref{this} corresponding to $T=4$ and $\xi=-\frac{\theta}{2\pi}$ with $\theta=0.7$.
{\em Right:} $\Xi_T$ with $\Theta_T$ superimposed. This illustrates the approximation of $\Theta_T$ by an affine lattice in fixed bounded subsets of the right half plane.} \label{fig3}
\end{figure}

\begin{figure}
\begin{center}
\includegraphics[width=0.49\textwidth]{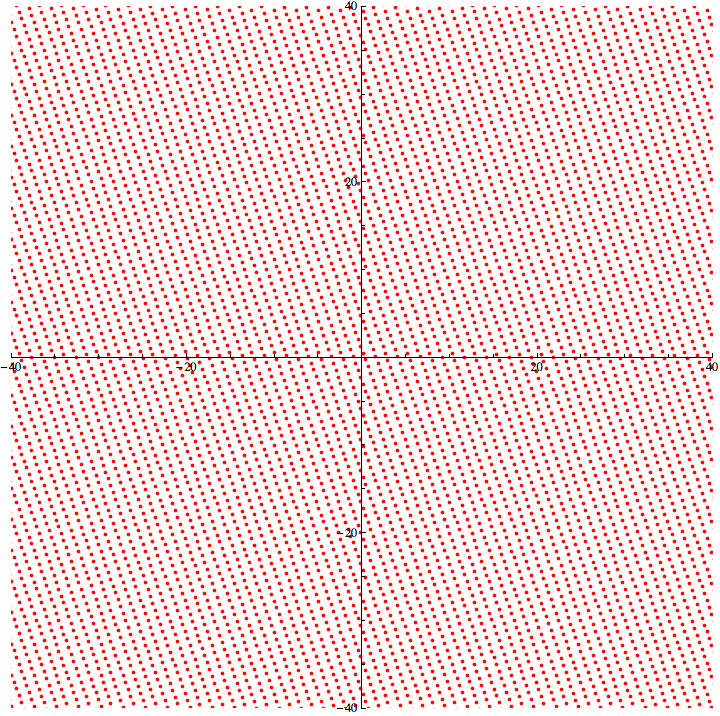}
\includegraphics[width=0.49\textwidth]{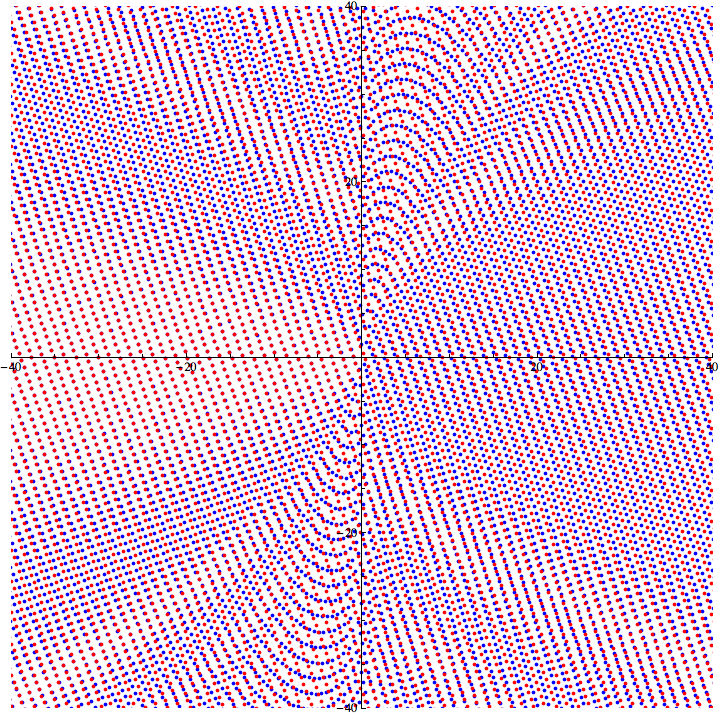}
\end{center}
\caption{{\em Left:}  A realization of the lattice $\widetilde\Xi_T$ in \eqref{this2} corresponding to $T=4$ and $\xi=\zeta+\frac12=-\frac{\theta}{2\pi}$ with $\theta=0.7$. {\em Right:}  $\widetilde\Xi_T$ with $\Theta_T$ superimposed. This illustrates the approximation of $\Theta_T$ by an affine lattice in fixed bounded subsets of the left half plane.} \label{fig5}
\end{figure}

Let us turn to $\sqrt n\bmod 1$. Consider the point set
\begin{equation}\label{P}
\scrP=\bigg\{ \bigg(\sqrt\frac{n}{\pi} \cos\big(2\pi\sqrt{n}\big), \sqrt\frac{n}{\pi} \sin\big(2\pi\sqrt{n}\big)\bigg) \,\bigg|\, n\in\NN \bigg\},
\end{equation}
see Figure \ref{fig1} (left), and the corresponding rotated (through an angle $\theta$) and stretched (by the linear map $D(T)$) point set
\begin{equation}\label{ThetaT}
\scrP k(\theta) D(T)=\bigg\{ \bigg(\sqrt\frac{n}{\pi T} \cos(2\pi\sqrt{n}-\theta), \sqrt{\frac{T n}{\pi}} \sin(2\pi\sqrt{n}-\theta)\bigg) \,\bigg|\, n\in\NN \bigg\},
\end{equation}
see Figure \ref{fig1} (right). For $\theta$ random, we denote by $\Theta_T=\scrP k(\theta) D(T)$ the corresponding random point set.
Note that $\scrP$ is a Delone set with uniform density in $\RR^2$ \cite{Delone}. That is, for any bounded $\scrA\subset\RR^2$ with boundary of Lebesgue measure zero,
\begin{equation}
\lim_{T\to\infty} \frac{\#(\scrP \cap T\scrA)}{T^2} = \vol\scrA.
\end{equation}
This follows from the fact that $(\sqrt n)_{n\in\NN}$ is uniformly distributed modulo one, since we can approximate $\scrA$ by finite unions and intersections of sectors of varying radii; see \cite{Delone} for details.

Let $\HH_-=\RR_{\leq 0}\times\RR$ and $\HH_+=\RR_{\geq 0}\times\RR$ denote the left and right half plane, respectively. We define the congruence subgroup
\begin{equation}
\Gamma_{2,0}(4)= \bigg\{ M\in\SLZ \,\bigg|\, M\equiv \begin{pmatrix} 1 & 0 \\ 0 & 1 \end{pmatrix} \text{or}  \begin{pmatrix} 1 & 2 \\ 0 & 1 \end{pmatrix} \bmod 4 \bigg\} ,
\end{equation}
and furthermore the random point set
\begin{equation}\label{ThetaDef}
\Theta = \big(\ZZ^2 g\cap\HH_+\big) \cup \big(-\big(\big[\ZZ^2+(\tfrac12,-\tfrac14)\big] g\big)\cap\HH_-\big),
\end{equation}
where $g$ is distributed according to the unique $G$-invariant probability measure $\mu$ on the homogeneous space $Y=\Lambda\backslash G$ with $\Lambda=\Gamma_{2,0}(4)\ltimes\ZZ^2$. Note that
\begin{equation}\label{well}
(\tfrac12,-\tfrac14) \Gamma_{2,0}(4) \in  (\tfrac12,-\tfrac14) +\ZZ^2,
\end{equation}
and hence $\big[\ZZ^2+(\tfrac12,-\tfrac14)\big] g$ is independent of the choice of representative of $\Lambda g$ in $Y$.

The purpose of this note is to prove the weak convergence $\Theta_T\Rightarrow\Theta$, and to investigate the properties of the limit process $\Theta$.

\begin{thm}\label{thm1}
If $\theta$ is random according to an absolutely continuous probability measure on $\RR/2\pi\ZZ$, then $\Theta_T\Rightarrow\Theta$ as $T\to\infty$.
\end{thm}

What we will in fact show is that for any bounded Borel sets $\scrA\subset\RR^2$ with boundary of Lebesgue measure zero and any integer $r$,
\begin{equation}\label{tcon}
\lim_{T\to\infty} \PP\big(\#(\Theta_T\cap \scrA)=r\big)
=\PP\big( \#(\Theta\cap \scrA)=r\big).
\end{equation}
Since $\Theta$ is a simple point process, the convergence in \eqref{tcon} implies the convergence in distribution asserted in Theorem \ref{thm1} \cite[Theorem 16.16]{Kallenberg02}. An important point in this argument is the observation that, by Theorem \ref{cor-5} below, we have $\vol\partial\scrA=0$ if and only if $\#(\Theta\cap \partial\scrA)=0$ almost surely.

We will see below (Theorem \ref{thm4}) that $\Theta$ is not invariant under the standard $\SLR$ action on $\RR^2$, although its two-point function is the same as that of a Poisson process (Theorem \ref{cor5}). Similar processes arise in the Boltzmann-Grad limit of the Lorentz gas in polycrystals \cite{poly}.

Note that Elkies and McMullen \cite{Elkies04} do not see the full process $\Theta$ since their setting corresponds to triangular test sets of the form $\{(x,y) \mid 0<x<1,\; |y|\leq \sigma x \}$, which are contained in the right half plane $\HH_+$. By definition \eqref{ThetaDef}, $\Theta$ is indistinguishable from the random affine lattice $\ZZ^2 g$ when restricted to the right half plane.

The plan for the proof of Theorem \ref{thm1} is to show that, in any bounded set $\scrA\subset\RR^2$, the set $\Theta_T$ is very close to an affine lattice, and then apply a slight extension of the Elkies-McMullen equidistribution theorem, which we will state now.

\section{Equidistribution}

Let 
\begin{equation}
N(\xi)=\bigg(\begin{pmatrix} 1 & 2\xi \\ 0 & 1 \end{pmatrix},(-\xi,-\xi^2)\bigg) .
\end{equation}
Note that the embedding $\RR\to G$, $\xi\mapsto N(\xi)$, defines a group homomorphism. We will in the following represent functions on $X=\GamG$ as functions on $G$ that are $\Gamma$-invariant, i.e., for which $f(\gamma g)=f(g)$ for all $\gamma\in\Gamma$. The integral of such a function over $X$ should be viewed as an integral restricted to the fundamental domain of $\Gamma$ in $G$.

\begin{thm}\label{thm2}
Let $\lambda$ be an absolutely continuous probability measure on $\TT=\RR/\ZZ$. Then, for any bounded continuous function $f:X\times X\to\RR$,
\begin{equation}
\lim_{T\to\infty}\int_\TT f\big(N(\xi)D(T),(1,(\tfrac12,-\tfrac14))N(\xi)D(T)\big) d\lambda(\xi)
= \int_{Y} f\big(g,(1,(\tfrac12,-\tfrac14))g\big) d\mu(g) . 
\end{equation}
\end{thm}

\begin{proof}
The space $Y$ is a finite cover of $X$. The equidistribution theorem \cite[Theorem 2.2]{Elkies04} extends to $Y$ (and in fact to any quotient by a finite-index subgroup of $\Gamma$): For any bounded continuous function $h:Y\to\RR$, we have 
\begin{equation}\label{equi}
\lim_{T\to\infty}\int_{\TT} h\big(N(\xi)D(T)\big) d\lambda(\xi) 
= \int_{Y} h(g) d\mu(g) .
\end{equation}
(Note that \cite[Theorem 2.2]{Elkies04} assumes $\lambda$ is the uniform probability measure on the interval $[0,p]$. The statement can be extended to general absolutely continuous probability $\lambda$ by a monotone class argument, where the density of $\lambda$ is approximated by finite linear combination of indicator functions of bounded intervals.)
To complete the proof, choose $h(g)=f\big(g, (1,(\tfrac12,-\tfrac14))g\big)$ and note that due to \eqref{well} this $h$ is indeed $\Lambda$-invariant (hence can be identified as a function on $Y$) and bounded continuous.
\end{proof}

Effective versions of the equidistribution theorem of \cite{Elkies04}, which is a consequence of Ratner's measure classification theorem \cite{Morris,Ratner91}, have recently been established in \cite{Browning13}.

\section{Properties of the limit process}

The intensity of a random point set $\Theta$ is defined by
\begin{equation}\label{siegel}
I_f(\Theta) =  \EE \sum_{\vecy\in\Theta} f(\vecy),
\end{equation}
where $f\in\C_0(\RR^2)$, i.e.\ continuous and with compact support. The following claim states that the intensity measure of our process is the Lebesgue measure.

\begin{thm}\label{cor-5}
For $f\in\L^1(\RR^2)$, 
\begin{equation}
I_f(\Theta) = \int_{\RR^2} f(\vecx) d\vecx .
\end{equation}
\end{thm}

\begin{proof}
The point processes defined by the point sets $\Xi = \ZZ^2 g$, $\widetilde\Xi=-([\ZZ^2+(\tfrac12,-\tfrac14)] g)$, with $g\in Y$ randomly distributed according to $\mu$, are translation-invariant and have asymptotic density equal to one. ($\Xi$ is in fact the same process as discussed in the introduction, despite the fact that we are now working with $\Lambda$ rather than $\ASL(2,\ZZ)$.) Therefore, by Campbell's formula, 
\begin{equation}\label{intm}
I_f(\Xi) = \int_{\RR^2} f(\vecx) d\vecx,\qquad I_f(\widetilde\Xi) = \int_{\RR^2} f(\vecx) d\vecx.
\end{equation}
Thus, for $f_\pm(\vecx)=f(\vecx)\chi_{\HH_\pm}(\vecx)$ the restriction of $f$ to the respective half plane,
\begin{equation}
I_f(\Theta) = I_{f_+}(\Xi) + I_{f_-}(\widetilde\Xi) =\int_{\RR^2} f_+(\vecx) d\vecx +\int_{\RR^2} f_-(\vecx) d\vecx =\int_{\RR^2} f(\vecx) d\vecx .
\end{equation}
\end{proof}

Theorem \ref{cor-5} implies the useful fact that for any Borel set $\scrA\subset\RR^2$ we have $\vol\partial\scrA=0$ if and only if $\#(\Theta\cap \partial\scrA)=0$ almost surely. To see this choose $f$ as the indicator function of $\partial\scrA$, which yields $\EE \,\#(\Theta\cap \partial\scrA) = \vol\partial\scrA$.

Let us now consider second-order correlations. 
For $f\in\C_0(\RR^2\times\RR^2)$, define the two-point function $R_f^\pm : Y \to \RR$ by
\begin{equation}
R_f^+(g)=\sum_{\vecm_1\neq \vecm_2\in\ZZ^2} f(\vecm_1g,\vecm_2g) ,
\qquad R_f^-(g)=\sum_{\vecm_1,\vecm_2\in\ZZ^2} f(\vecm_1g,-[\vecm_2(1,(\tfrac12,-\tfrac14))g]) .
\end{equation}
This defines linear functionals $f\mapsto R_f^\pm$ from $\L^1(\RR^2\times\RR^2)$ to $\L^1(Y,d\mu)$.

\begin{thm}\label{thm3}
For $f\in\L^1(\RR^2\times\RR^2)$,
\begin{equation}
\int_Y R_f^\pm(g) d\mu(g)= \int_{\RR^2\times\RR^2} f(\vecx,\vecy) d\vecx d\vecy.
\end{equation}
\end{thm}

\begin{proof}
The relation for $R_f^+$ is proved in \cite[Proposition A.3]{ElBaz15b}. The other is analogous: In view of the density of $\C_0$ in $\L^1$ and the Lebesgue monotone convergence theorem, it suffices to prove the claim for $f\in\C_0(\RR^2\times\RR^2)$. Let $\mu_0$ denote the $\SLR$-invariant probability measure on $Y_0=\Gamma_{2,0}(4)\backslash\SLR$. Then
\begin{equation}
\begin{split}
\int_Y R_f^-(g) d\mu(g) 
& = \int_{Y_0} \int_{\TT^2}
\sum_{\vecm_1,\vecm_2} f((\vecm_1+\vecy)M,-(\vecm_2+\vecy+(\tfrac12,-\tfrac14))M) d\vecy d\mu_0(M)
 \\
& = \int_{Y_0} \int_{\RR^2}
\sum_{\vecm} f(\vecy,-(\vecm+(\tfrac12,-\tfrac14))M-\vecy) d\vecy d\mu_0(M) .
\end{split}
\end{equation}
The Siegel integral formula for $\tilde f\in\C_0(\RR^2)$ reads
\begin{equation}
\int_{Y_0} \sum_{\vecm\in\ZZ^2} \tilde f((\vecm+(\tfrac12,-\tfrac14))M) d\mu_0(M) = \int_{\RR^2} \tilde f(\vecx) d\vecx.
\end{equation}
This follows either by direct computation or the general Siegel-Veech formula \cite{Veech}; it is in fact also a special case of the generalized Siegel-Veech formula for Euclidean model sets \cite[Theorem 5.1]{quasi}.
Applying this with 
\begin{equation}
\tilde f(\vecx)= \int_{\RR^2} f(\vecy,-\vecx-\vecy) d\vecy
\end{equation}
proves the theorem.
\end{proof}

The two-point intensity of the random point process $\Theta$ is defined by
\begin{equation}
K_f(\Theta)= \EE \sum_{\vecy_1\neq\vecy_2\in\Theta} f(\vecy_1,\vecy_2),
\end{equation}
where $f\in\C_0(\RR^2\times\RR^2)$.
The following corollary of Theorem \ref{thm3} shows that the the two-point intensity of $\Theta$ is the same as that of a Poisson process. This extends the observation of \cite{ElBaz15} for $\sqrt n\bmod 1$ to the full two-dimensional process.

\begin{thm}\label{cor5}
For $f\in\L^1(\RR^2\times\RR^2)$, 
\begin{equation}
K_f(\Theta) = \int_{\RR^2\times\RR^2} f(\vecx_1,\vecx_2) d\vecx_1 d\vecx_2 .
\end{equation}
\end{thm}

\begin{proof}
Let 
\begin{align}
f_{++}(\vecx_1,\vecx_2) & =f(\vecx_1,\vecx_2) \chi_{\HH_+}(\vecx_1)  \chi_{\HH_+}(\vecx_2), &
f_{+-}(\vecx_1,\vecx_2) & =f(\vecx_1,\vecx_2) \chi_{\HH_+}(\vecx_1)  \chi_{\HH_-}(\vecx_2), \\
f_{-+}(\vecx_2,\vecx_1)&=f(\vecx_1,\vecx_2) \chi_{\HH_-}(\vecx_1)  \chi_{\HH_+}(\vecx_2), &
f_{--}(\vecx_1,\vecx_2)&=f(\vecx_1,\vecx_2) \chi_{\HH_-}(\vecx_1)  \chi_{\HH_-}(\vecx_2),
\end{align}
be the restrictions to the various half planes. We have
\begin{equation}
K_f(\Theta) = \int_Y \big[R_{f_{++}}^+(g) + R_{f_{--}}^+((1,(\tfrac12,-\tfrac14))g) + 
R_{f_{+-}}^-(g) + R_{f_{-+}}^-(g) \big] d\mu(g).
\end{equation}
The corollary now follows from Theorem \ref{thm3} by integrating term-wise.
\end{proof}

Given a subgroup $H\subset G$, we say $\Theta$ is $H$-invariant if $\Theta h$ has the same distribution as $\Theta$ (viewed as random point sets in $\RR^2$) for all $h\in H$.

Consider the subgroup
\begin{equation}
P= \bigg\{ \bigg(\begin{pmatrix} a & b \\ 0 & 1/a \end{pmatrix}, (0,y) \bigg) \,\bigg|\, b,y\in\RR, \; a\in\RR\setminus\{0\} \bigg\} \subset G.
\end{equation}

\begin{thm}\label{thm4}
$\Theta$ is $P$-invariant but not $\SLR$-invariant.
\end{thm}

\begin{proof}
The invariance under 
\begin{equation}
\bigg(\begin{pmatrix} a & b \\ 0 & 1/a \end{pmatrix}, (0,y) \bigg)
\end{equation}
is evident for all $a>0$, $b\in\RR$, since these transformations preserve the half planes $\HH_+$ and $\HH_-$.
What remains for the proof of $P$-invariance is the invariance under the element $(-1,(0,0))$, which acts on $\RR^2$ by reflection at the origin. By the $G$-invariance of $\mu$, $\ZZ^2 g$ has the same distribution as $\ZZ^2(-1,(\tfrac12,-\tfrac14))g$. This implies that $\Theta = \big(\ZZ^2 g\cap\HH_+\big) \cup \big(-\big(\big[\ZZ^2+(\tfrac12,-\tfrac14)\big] g\big)\cap\HH_-\big)$ has the same distribution as the random point set 
\begin{equation}
\begin{split}
\Theta' 
& = \big(\ZZ^2 (-1,(\tfrac12,-\tfrac14)) g\cap\HH_+\big) \cup \big(-\big(\big[\ZZ^2+(\tfrac12,-\tfrac14)\big] (-1,(\tfrac12,-\tfrac14)) g\big)\cap\HH_-\big) \\
& = \big(\big[\ZZ^2+(\tfrac12,-\tfrac14)\big]g\cap\HH_+\big) \cup \big(-\big(\ZZ^2 g\big)\cap\HH_-\big).
\end{split}
\end{equation}
This shows $\Theta=-\Theta'$, which establishes the $(-1,(0,0))$-invariance.

As to the $\SLR$-invariance, we note that every realisation of $\Theta$, restricted to the right half planes $\HH_+$, looks like a affine lattice $\ZZ^2$ restricted to the right half plane. It is evident that the rotated process 
$$
\widetilde\Theta= \Theta \begin{pmatrix} 0 & -1 \\ 1 & 0 \end{pmatrix} 
$$
does not have this property, since every realisation produces two different affine lattices (which are copies of the same underlying lattice) in the upper and lower half plane, respectively. This shows that $\widetilde\Theta$ and $\Theta$ do not have the same distribution. Hence $\Theta$ is not $\SLR$-invariant.
\end{proof}

\section{Proof of Theorem 1} 


To establish the convergence $\Theta_T\Rightarrow \Theta$, it is sufficient to prove that convergence holds in finite-dimensional distribution \cite[Theorem 16.16]{Kallenberg02} (recall the remark following Theorem \ref{cor-5}). 
Since $\Theta$ is a simple point process, it is in fact sufficient to consider the one-dimensional distributions. 
That is, we need to prove that {\em for any bounded Borel set $\scrA\subset\RR^2$ with boundary of Lebesgue measure zero, the random variable $\#(\Theta_T\cap \scrA)$ converges in distribution to $\#(\Theta\cap \scrA)$.} It is in fact sufficient to prove the convergence for test sets $\scrA$ that are rectangles of the form $[a,b]$. More general Borel sets $\scrA$ (bounded with boundary of Lebesgue measure zero) can then be approximated by unions of such rectangles. This will require proof of convergence for the joint distribution for finitely many rectangles. We will limit the presentation to one rectangle; the case of multiple rectangles is analogous.

{\em The right half plane.} The following estimates are almost identical to those in \cite[Section 3]{Elkies04}, which considers the case of triangular test sets of the form $\{(x,y) \mid 0<x<1,\; |y|\leq \sigma x \}$, rather than general rectangles. Consider a rectangle $[a,b]\times[c,d]$ and assume for now $a\geq 0$. 
Note first of all that in \eqref{ThetaT} the sine has to be of order $O((T n)^{-1/2})$, and thus its argument must be close, by the same order, to $0$ or $\pi$ mod $2\pi$. If it is close to $\pi$, the cosine is negative which is ruled out by the assumption $a\geq 0$. 
Set $\xi=-\theta/2\pi$, and define $m\in\ZZ$ so that $-\frac12\leq\sqrt{n}+\xi+m<\frac12$. Using $4|x| \leq | \sin(2\pi x) |$ for $|x|\leq \frac14$, we have for $Tn$ sufficiently large,
\begin{equation}
4 |\sqrt{n}+\xi+ m| \leq | \sin(2\pi(\sqrt{n}+\xi)) | \leq \max\{|c|,|d|\} \sqrt{\frac{\pi}{T n}}  
\end{equation}
and thus
\begin{equation}\label{ads}
\sqrt n = -(m+\xi) + O(1/\sqrt{Tn}) .
\end{equation}
The objective is now to linearise the inequalities
\begin{equation}
a\leq \sqrt\frac{n}{\pi T} \cos(2\pi(\sqrt{n}+\xi))\leq b, \qquad 
c\leq \sqrt{\frac{T n}{\pi}} \sin(2\pi(\sqrt{n}+\xi)) \leq d.
\end{equation}
Taylor's theorem tells us that $|\cos(x)-1|\leq \tfrac12 x^2$ and $|\sin(x)-x|\leq \tfrac16 |x|^3$, and so 
\begin{equation}\label{there}
a  \leq \sqrt\frac{n}{\pi T} +  O\bigg(\frac{1}{T^{3/2} n^{1/2}}\bigg) \leq b ,
\end{equation}
\begin{equation}\label{here}
\frac{c}{2\sqrt{\pi T n}}  \leq \sqrt{n}+\xi+m + O\bigg(\frac{1}{(T n)^{3/2}}\bigg) \leq \frac{d}{2\sqrt{\pi T n}} .
\end{equation}
The last inequality transforms to
\begin{equation}
\frac{c}{2\sqrt{\pi T n}} -(m+\xi)\leq \sqrt{n} + O\bigg(\frac{1}{(T n)^{3/2}}\bigg) \leq \frac{d}{2\sqrt{\pi T n}}-(m+\xi),
\end{equation}
which is equivalent to 
\begin{equation}\label{eq:ten}
\bigg[\frac{c}{2\sqrt{\pi T n}} -(m+\xi)\bigg]^2 \leq n + O\bigg(\frac{1}{T^{3/2} n}\bigg) \leq \bigg[\frac{d}{2\sqrt{\pi T n}}-(m+\xi)\bigg]^2.
\end{equation}
Now \eqref{eq:ten} is equivalent to
\begin{equation}\label{here0}
-\frac{c(m+\xi)}{\sqrt{\pi T n}} +(m+\xi)^2 \leq n + O\bigg(\frac{1}{Tn}\bigg) \leq -\frac{d(m+\xi)}{\sqrt{\pi T n}} +(m+\xi)^2  .
\end{equation}
In view of \eqref{ads}, this is equivalent to
\begin{equation}\label{eq:twe}
\frac{c}{\sqrt{\pi T}} +(m+\xi)^2  \leq n + O\bigg(\frac{1}{Tn}\bigg)  \leq \frac{d}{\sqrt{\pi T}} +(m+\xi)^2  ,
\end{equation}
and furthermore \eqref{there} is equivalent to
\begin{equation}\label{there2}
 a   \leq -\frac{m+\xi}{\sqrt{\pi T}}   +  O\bigg(\frac{1}{T n^{1/2}}\bigg) \leq b  .
\end{equation}
Let us assume that $\xi\notin [-\delta,\delta]+\ZZ$, where $\delta>0$ may be chosen arbitrarily small. (This assumption is without loss of generality, because the event $\xi\in [-\delta,\delta]+\ZZ$ has probability at most $\lambda([-\delta,\delta])$, which tends to zero as $\delta\to 0$.) We can then drop the condition $n\geq 1$, as the positivity of $n$ is implied in \eqref{eq:twe} for $T$ sufficiently large.
Next, replacing $n$ by $n+m^2$ yields
\begin{equation}\label{rightineq}
a  \leq -\frac{m+\xi}{\sqrt{\pi T}}  + O\bigg(\frac{1}{T}\bigg) \leq b,\qquad \frac{c}{\sqrt{\pi T}}  \leq n -2m\xi - \xi^2 + O\bigg(\frac{1}{T}\bigg) \leq \frac{d}{\sqrt{\pi T}}  .
\end{equation}
So, given any $\epsilon>0$ there exists $T>0$ such that the number of points $(m,n)\in\ZZ^2$ satisfying \eqref{rightineq} is bounded above (resp.\ below) by the number of points with
\begin{equation} \label{OT0}
\bigg(-\frac{m+\xi}{\sqrt{\pi T}},\sqrt{\pi T}(n -2m\xi - \xi^2)\bigg) \in \text{$\scrA_\epsilon^+$ (resp.\  $\scrA_\epsilon^-$)} 
\end{equation}
with 
\begin{equation}
\scrA_\epsilon^+=[a-\epsilon,b+\epsilon]\times[c-\epsilon,d+\epsilon], \qquad  
\scrA_\epsilon^-=[a+\epsilon,b-\epsilon]\times[c+\epsilon,d-\epsilon]. 
\end{equation}
Note that \eqref{OT0} can be expressed as
\begin{equation} \label{OT}
(-m,n) N(\xi) D(\pi T) \in \text{$\scrA_\epsilon^+$ (resp.\  $\scrA_\epsilon^-$)}  .
\end{equation}
We have thus established that in rectangles with $a\geq 0$ the process $\Theta_T$ looks asymptotically like the random affine lattice
\begin{equation}\label{this}
\Xi_T=\ZZ^2N(\xi) D(\pi T) ,
\end{equation}
see Figure \ref{fig3}. We consider $\Xi_T$ as a process in $\RR^2$, so the fact that $\scrA_\epsilon^+$ might intersect with the left half plane (e.g.\ in the case $a=0$) is not a cause for concern.
Elkies and McMullen's work implies that for $\xi$ random with respect to any absolutely continuous probability measure on $\RR/\ZZ$, the process $\Xi_T$ converges in distribution to $\Xi$, as $T\to\infty$, where $\Xi$ is the point process defined at the start of Section \eqref{RPS}. In view of \eqref{intm}, the probability of having one or more points in $\scrA_\epsilon^+\setminus\scrA_\epsilon^-$ is of order $\epsilon$. The limits $T\to\infty$ and $\epsilon\to 0$ therefore commute. This would complete the proof if we were only interested in test sets in the right half plane. 

{\em The left half plane.} Let us therefore turn to the case $b\leq 0$. In this case the argument of the sine has to be close to $\pi$ as the cosine is now required to be negative. Then \eqref{here} is replaced by
\begin{equation}\label{here2}
-\frac{d}{2\sqrt{\pi T n}}  \leq \sqrt{n}+\zeta+m + O\bigg(\frac{1}{(T n)^{3/2}}\bigg) \leq -\frac{c}{2\sqrt{\pi T n}} .
\end{equation}
with $\zeta=-(\theta+\pi)/2\pi=\xi-1/2$. Repeating the steps in the previous calculation leads to
\begin{equation}\label{rightineq2}
a  \leq \frac{m+\zeta}{\sqrt{\pi T}}  + O\bigg(\frac{1}{T}\bigg) \leq b,\qquad \frac{c}{\sqrt{\pi T}}  \leq -n +2m\zeta + \zeta^2 + O\bigg(\frac{1}{T}\bigg) \leq \frac{d}{\sqrt{\pi T}}  .
\end{equation}
\begin{equation}
-\frac{d}{\sqrt{\pi T}}  \leq n -2m\zeta - \zeta^2 + O\bigg(\frac{1}{T}\bigg) \leq -\frac{c}{\sqrt{\pi T}} ,\qquad
a\sqrt{\pi T}  \leq m+\zeta + O\bigg(\frac{1}{T}\bigg) \leq b\sqrt{\pi T}.
\end{equation}
So upper (resp.\ lower) bounds on the number of points are given by
\begin{equation}
\bigg(\frac{m+\zeta}{\sqrt{\pi T}},\sqrt{\pi T}(-n +2m\zeta + \zeta^2)\bigg) \in \text{$\scrA_\epsilon^+$ (resp.\  $\scrA_\epsilon^-$)} ,
\end{equation}
and hence
\begin{equation}
(m,-n) \bigg(\begin{pmatrix} 1 & 2\zeta \\ 0 & 1 \end{pmatrix},(\zeta,\zeta^2)\bigg) D(\pi T) \in \text{$\scrA_\epsilon^+$ (resp.\  $\scrA_\epsilon^-$)}  .
\end{equation}
We have thus shown that in rectangles with $b\leq 0$ the process $\Theta_T$ looks like the affine lattice
\begin{equation}\label{this1}
\ZZ^2\bigg(\begin{pmatrix} 1 & 2\zeta \\ 0 & 1 \end{pmatrix},(\zeta,\zeta^2)\bigg) D(\pi T)
= -\big[ \ZZ^2 N(\zeta) D(\pi T) \big].
\end{equation}
Finally we check that
\begin{equation}
\ZZ^2 N(\zeta)
=  \ZZ^2 \bigg(\begin{pmatrix} 1 & -1 \\ 0 & 1 \end{pmatrix},(\tfrac12,-\tfrac14)\bigg)  N(\xi) \\
=  \ZZ^2 \big(1,(\tfrac12,-\tfrac14)\big)  N(\xi) .
\end{equation}
We have now established that in rectangles with $b\leq 0$ the process $\Theta_T$ looks asymptotically like the random affine lattice
\begin{equation}\label{this2}
\widetilde\Xi_T=-[\ZZ^2 \big(1,(\tfrac12,-\tfrac14)\big) N(\xi) D(\pi T)],
\end{equation}
see Figure \ref{fig5}. As in the case of the right half plane, the process $\widetilde\Xi_T$ converges in distribution to $\tilde\Xi$ (which has the identical distribution as the process $\Xi$), as $T\to\infty$, provided we restrict test sets to rectangles in the left half plane. The final challenge is now to combine these results to establish joint convergence in both half planes.

{\em Joint convergence.} Let us now consider the remaining case $a\leq 0 \leq b$. We decompose the rectangle as
\begin{equation}
[a,b]\times[c,d]= ([a,0]\times[c,d]) \cup ([0,b]\times[c,d]).
\end{equation}
The argument for the left and right halfplanes show that an upper bound for the number of points is obtained by considering, for any $\epsilon>0$, the joint distribution of
\begin{equation}\label{jtXi}
\widetilde\Xi_T \cap ([a-\epsilon,\epsilon]\times[c-\epsilon,d+\epsilon]) , \qquad \Xi_T \cap ([-\epsilon,b+\epsilon]\times[c-\epsilon,d+\epsilon]),
\end{equation}
where $\widetilde\Xi_T=\ZZ^2 \big(1,(\tfrac12,-\tfrac14)\big) N(\xi) D(\pi T)$ and $\Xi_T=\ZZ^2 N(\xi) D(\pi T)$ are not independent, since the are functions of the same random varaible $\xi$. Theorem \ref{thm2} implies that the joint limit distribution of \eqref{jtXi} is given by
\begin{equation}\label{jtXi2}
\widetilde\Xi \cap ([a-\epsilon,\epsilon]\times[c-\epsilon,d+\epsilon]) , \qquad \Xi \cap ([-\epsilon,b+\epsilon]\times[c-\epsilon,d+\epsilon]),
\end{equation}
where $\widetilde\Xi=[\ZZ^2+(\tfrac12,-\tfrac14)] g$, $\Xi = \ZZ^2 g$, and $g\in Y$ randomly distributed according to $\mu$. Similarly, a lower bound on the number of points is given by
\begin{equation}\label{jtXi3}
\widetilde\Xi \cap ([a+\epsilon,-\epsilon]\times[c+\epsilon,d-\epsilon]) , \qquad \Xi \cap ([\epsilon,b-\epsilon]\times[c+\epsilon,d-\epsilon]).
\end{equation}
The remark following Theorem \ref{cor-5} provides the regularity that allows us to exchange the limits $T\to\infty$ and $\epsilon\to 0$. This shows that indeed 
$\Theta_T \cap ([a,b]\times[c,d])$ converges in distribution to 
\begin{equation}
[\widetilde\Xi \cap ([a,0]\times[c,d])] \cup [\Xi \cap ([0,b]\times[c,d])] = \Theta \cap ([a,b]\times[c,d]) ,
\end{equation}
as required.
\qed

\end{document}